\documentclass{amsart}
\usepackage{amssymb}
\newcommand{\Q}{\mathbb Q}
\newcommand{\C}{\mathbb C}
\newcommand{\R}{\mathbb R}
\newcommand{\Z}{\mathbb Z}
\newcommand{\N}{\mathbb N}

\newcommand{\p}{\mathfrak p}
\newcommand{\fP}{\mathfrak P}
\newcommand{\OO}{\mathcal O}
\newcommand{\disc}{\operatorname{disc}}
\newcommand{\diff}{\operatorname{diff}}
\newcommand{\Orb}{\operatorname{Orb}\,}
\newcommand{\eps}{\varepsilon}
\newcommand{\rsp}{\raisebox{0em}[2.7ex][1.3ex]{\rule{0em}{2ex} }}

\newtheorem{prop}{Proposition}
\newtheorem{theorem}{Theorem}

\begin{document}

\title{Euclid's Algorithm in quartic CM-fields}
\author{Franz Lemmermeyer}
\address{M\"orikeweg 1, 73489 Jagstzell}
\email{hb3@ix.urz.uni-heidelberg.de}

\abstract
In this note we present techniques to compute inhomogeneous minima
of norm forms; as an application, we determine all norm-Euclidean 
complex bicyclic quartic number fields.
\endabstract

\maketitle

A number field K is said to be Euclidean (with respect 
to the norm), if for all $\xi \in K$ we can find 
$\eta \in \OO_K$ such that $|N_{K/\Q}(\xi-\eta)|<1$. 
Although it is known since the work of Davenport that there 
are only a finite number of Euclidean fields with unit rank 
$1$, only the quadratic Euclidean fields have been 
determined so far. In this paper, we will determine the 
Euclidean normal quartic CM-fields (these are totally complex quartic 
fields which contain a real quadratic subfield). According to a 
well known theorem due  to  Cassels \cite{Cas}, such fields have 
discriminants $ < 230\,202\,117$. In fact, the bound given by 
Cassels was somewhat smaller, but his computations were shown to 
contain an error by van der Linden \cite{vdL}. Using Setzer's 
solution of the class number 1 problem for complex cyclic 
quartic number fields, van der Linden was able to prove

\begin{theorem}\label{T1}
There are exactly two complex cyclic quartic fields that are
norm-Euclidean: $\Q(\zeta_{5})$ and the quartic subfield of 
$\Q(\zeta_{13})$.
\end{theorem} 

He also gave bounds for $\disc \, K$  in case $K$ is a complex bicyclic 
quartic field, but did not attempt to determine them all.
Making use of ideas of Sauvageot \cite{Sau}, we will prove

\begin{theorem}\label{T2}
The Euclidean fields $\Q(\sqrt{-m},\sqrt{n}\,)$, $m \in \N, 
n \in \Z$, are given by
$$\begin{array}{rlll}
 m = & 1, &n = &2, 3, 5, 7;  \\
 m = & 2, &n = &-3, 5;           \\
 m = & 3, &n = &2, 5, -7, -11, 17, -19;\\
 m = & 7, &n = &5.
\end{array}$$
\end{theorem} 

\section{Lower bounds for Euclidean minima}

We begin by fixing the notation. For an algebraic number field $K$,
$\OO_K$ denotes its ring of integers, and $E_K$ its unit 
group. For an ideal $\mathfrak a$ in $\OO_K$,   $N \mathfrak a$ will
always denote the absolute norm of $\mathfrak a$, i.e. the index 
$(\OO_K:\,\mathfrak a)$. For $\xi \in K$, put
$$M(\xi,K) = \inf \ \{|N_{K/\Q}(\xi-\alpha)|\,:\,\alpha \in \OO_{K}\};$$
$M(\xi,K)$ is called the Euclidean minimum of $K$ at $\xi$
(it can be proved that $M(\xi,K)$ is in fact a minimum; cf. 
\cite{BSD1}, \cite{BSD2} or \cite{Lem2}); obviously, $K$ is Euclidean 
if and only if $M(\xi,K)<1$ for all $\xi \in K$. Moreover,
$$M(K) = \sup \ \{M(\xi,K): \xi \in K\}$$
is called the Euclidean minimum of $K$; we know that this supremum 
is a maximum if $K$ has unit rank $\le 1$, and we conjecture that 
this holds for all number fields.

There are three simple methods that allow us to prove that a 
given field is not Euclidean: the use of ramified primes, the 
residue classes modulo ideals of small norm, and the use of 
absolute values. These techniques have been used to determine 
all quadratic Euclidean fields, and their usefulness has been 
stressed again by Cioffari \cite{Cio} in his determination of 
all pure cubic Euclidean fields.

\begin{prop}
Let $K/k$ be a finite extension of number fields of relative degree $n$, 
and suppose that the prime ideal $\p$ in $\OO_K$  is completely 
ramified in $K/k$, i.e. that $\p\OO_K =\,\mathfrak P^n$. 
If $\beta \equiv \alpha^n  \bmod \p$ for some 
$ \alpha,\beta \in\,\OO_K \setminus \p$, 
and if there do not exist $b \in \OO_K$ such that
\begin{enumerate}
\item $b \equiv \beta \,\bmod \p$;
\item $b = N_{K/k}\delta$ for some $\delta \in \OO_K$;
\item $|N_{k/\Q}\,b| < N\p$;
\end{enumerate}
then K is not Euclidean.
\end{prop}

\begin{proof} Suppose that $K$ is Euclidean; then there is a
$\pi \in \OO_K$ such that $\mathfrak P\,=\pi\OO_K$,
and for $\xi= \alpha/\pi$ we can find $\eta \in \OO_K$  
such that $|N_{K/\Q}(\xi-\eta)| < 1$. This implies
$|N_{K/\Q}(\alpha-\eta \pi)| < N\mathfrak P$; put 
$b= N_{K/k}(\alpha-\eta \pi)$. Then we find

\smallskip\noindent
(1) $b \equiv \beta \bmod \p$, because $\alpha-\eta \pi 
    \equiv \alpha \bmod \,\mathfrak P$ and the fact that $\p$ is 
    completely ramified in $K/k$ imply that 
    $N_{K/k}(\alpha-\eta \pi) \equiv N_{K/k}\alpha \bmod \fP$ 
    as a congruence in the normal closure of $K/k$. Since both sides 
    are $\in \OO_K$, the congruence holds $ \bmod \fP \cap \OO_K = \,\p$. 

\smallskip\noindent
(2) $b = N_{K/k} \delta$ for $\delta = \alpha - \eta \pi$ is clear; 

\smallskip\noindent
(3) $|N_{k/\Q}b| = |N_{k/\Q}N_{K/k}(\alpha - \eta \pi)|\
       = |N_{K/\Q}(\alpha- \eta \pi)| < N_{K/\Q}\mathfrak P 
       = N_{k/\Q}\,\p$. 
\end{proof}

In the special case $k=\Q$ and $\p=p\Z$, there are 
only two $b \in \Z$ satisfying (1) and (3), because
$|N_{k/\Q}\,b| = |b|$ and $|N\p| = p$. Moreover, if $K$ is 
totally complex, only positive $b \in \Z$ can be norms from $K$.

We note that we can use a modification of Proposition 1
to determine lower bounds for $M(\xi,K)$; but this will not be 
needed in the sequel. Moreover, there is an immediate generalization 
to products of pairwise different completely ramified prime ideals.

The idea behind our next result is due to Barnes and Swinnerton-Dyer
(BSD). Let $\xi=\xi_{1} \in K$ and $\eps \in E_K$
be given; it is easy to see that there is an $m \in \N$
such that $\eps^m \xi-\xi \in \OO_K$  (we will 
often write $\xi \equiv \eta \,\bmod \OO_K$  for 
$\xi-\eta \in \OO_K$). The set 
$ \Orb_\eps(\xi) = \{\xi=\xi_0, \xi_1,\ldots,\xi_{\ell-1}: 
	\xi_{j} \equiv \eps^j \xi \bmod \OO_K, 1 \le j \le \ell\}$  
of representatives $\bmod \ \OO_K$ of the $\eps^j \xi$ is 
called the {\it orbit} of $\xi$. It is clear from the definition of 
the Euclidean minimum that $M(\xi_0,K) = ... = M(\xi_{\ell-1},K)$
for all $\xi_i \in \Orb_\eps(\xi)$.

\begin{prop} 
Let $K=\Q(\sqrt{m}, \sqrt{n}\,)$ be an imaginary bicyclic number 
field, $\xi \in K$, and suppose that  
$\{\xi = \xi_0,\ldots,\xi_{\ell-1}\} = \Orb_\eps(\xi)$
for a unit $\eps \in \OO_K^\times$, where $|\eps|>1$ for some
fixed embedding $|\cdot|$ of $K$ into $\C$. If
$M(\xi,K)<\kappa$ for some $\kappa \in \R$, then there is an element
$\alpha = r_1+r_2\sqrt{m}+r_3\sqrt{n}+r_4\sqrt{mn} \in K$
with the following properties:
\begin{enumerate}
 \item $\alpha \equiv \xi_j \, \bmod \OO_K$  
        for some $0 \le j \le \ell-1$; 
 \item $|N_{K/\Q}\,\alpha| < \kappa$; 
 \item $|r_i| \le \mu_i$  for $ 1 \le i \le 4$,
       where the bounds $\mu_i$ are defined by
       $$\begin{array}{ll}
       \mu_1=\frac12 \sqrt[4\,]{\kappa}\Bigl(\sqrt{|\eps|}
        +1/\sqrt{|\eps|}\Bigr),
       & \mu_2=\mu_1/\sqrt{|m|}, \\
         \mu_3=\mu_1/\sqrt{|n|}, 
       & \mu_4=\mu_1/\sqrt{|mn|}. %
       \end{array}$$
\end{enumerate}
\end{prop}

\begin{proof} Assume that $M(\xi,K)<\kappa$; then
$|N_{K/\Q}(\xi-\eta)| < \kappa$ for some $\eta \in \OO_K$. 
Since $|\eps|>1$, we can find $n \in \N$ such that
$$\sqrt[4\,]{k} \cdot \sqrt{|\eps|\,}\,^{-1} \le
 |(\xi-\eta)\eps^n| < \sqrt[4\,]{\kappa} \cdot
 \sqrt{|\eps|}.$$
If we let $\alpha = (\xi-\eta)
 \eps^n$, $\alpha$ will satisfy conditions 1. (with 
$j \equiv n \bmod \ell$) and 2. Now
$$4|r_1| = |\alpha + \alpha'+\alpha''+\alpha'''| \le
  2|\alpha| + 2|\alpha'|,$$
where $\alpha,\,\alpha',\,\alpha'',\,\alpha'''$, are the
conjugates of $\alpha$, and where $\alpha''$ denotes the
complex conjugate of $\alpha$ (this implies $|\alpha''|=|\alpha|$).
The inequality 
\begin{align*}
0 &< (\sqrt[4\,]{\kappa} \cdot \sqrt{|\eps|}-|\alpha|)
      (\sqrt[4\,]{\kappa} \cdot \sqrt{|\eps|}-|\alpha'|) \\
  &= \sqrt{\kappa} \cdot |\eps| - \sqrt[4\,]{\kappa} \cdot 
     \sqrt{|\eps|}\,(|\alpha|+|\alpha'|) +|\alpha \alpha'|,
\end{align*}
together with $|\alpha \alpha'|^2 = N_{k/\Q} \alpha < \kappa $ 
yields $4|r_1| < 4\mu_1$. Similarly, we get
$$ 4|r_2\sqrt{m}| = |\alpha - \alpha' + \alpha'' - \alpha'''| 
   \le 2|\alpha| + 2|\alpha'| < 4\mu_1 \text{ etc.,}$$
if we assume that $\sqrt{m}$ is fixed by complex conjugation, 
i.e. that $m>0$. In case $m<0$, we have to switch some signs in 
$|\alpha - \alpha' + \alpha'' - \alpha'''|$, but this does not change
the resulting bound.
\end{proof}

Propositions \ref{P3} and \ref{P4} below will not be 
needed for the proofs of Theorems 1 and 2; they are included 
because they might turn out to be useful in the determination 
of Euclidean CM-fields of higher degree.

\begin{prop}\label{P3}
Let $L$ be a CM-field with maximal real subfield $K$; 
if $L$ is norm-Euclidean, but $K$ is not, then 
$N_{K/\Q}\, \disc(L/K) < 4^{(K:\Q)}.$ 
\end{prop}

\begin{proof} Suppose that $L$ is Euclidean; for every $\xi \in K$ 
we can find $\eta \in \OO_L$  such that $N_{L/\Q}(\xi-\eta) < 1$. 
Let $\sigma$ denote complex conjugation; then
$$ \begin{array}{rcl} 
   N_{L/K}(\xi-\eta) & = & (\xi-\eta)(\xi-\eta^\sigma)             
         =  \xi^2-\xi(\eta + \eta^\sigma)+\eta\eta^\sigma \\  
         & = & \frac14 \left((2\xi-T_{L/K}\eta)^{2} - 
                    (\eta - \eta^{\sigma})^2\right).\\
         & = & \frac14 \left(\left(2\xi-T_{L/K}\eta\right)^{2} +
			N_{L/K}(\eta-\eta^\sigma)\right).
\end{array}$$
Choose $\xi \in K$ with $M(\xi,K) \ge 1$; this implies 
$\eta \ne \eta^\sigma$, because otherwise 
$$  N_{L/\Q}(\xi-\eta) = N_{K/\Q}(\xi-\eta)^2 \ge M(\xi,K) \ge 1.$$   
Therefore, $0 \neq \eta - \eta^\sigma \in \,\diff(L/K)$, i.e. 
$\diff(L/K)|(\eta - \eta^{\sigma})$. Hence
\begin{align*} 
  N_{L/K}(\xi - \eta) &\ge 
  \frac14\left(N_{L/K}(\eta - \eta^{\sigma})\right), \\
  1 > N_{L/\Q}(\xi - \eta) 
  &\ge 4^{-(K:\Q)}N_{K/\Q}\left(N_{L/K}(\eta - \eta^{\sigma})\right)
  \ge 4^{-(K:\Q)}N_{L/\Q}\,\text{diff}\,(L/K),
\end{align*}
and the asserted inequality follows from 
disc$\,(L/K) = N_{L/K}\,$diff$\,(L/K)$.
\end{proof}

\begin{prop}\label{P4}                
Let $L$ be a $CM$-field  with maximal real subfield $K$; if 
$\diff(L/K) \equiv 0 \bmod 2$, then $M(L) \ge M(K)^2$.
\end{prop}

This result is best possible: for $L = \Q(\zeta_{12})$ and 
$K = \Q(\sqrt{3}\,)$ we actually have equality since $M(K) = \frac12$
and $M(L) = \frac14$. 

\begin{proof} 
If $\diff(L/K) \equiv 0 \bmod 2$, then  $2 \mid T_{L/K}\eta$. Moreover, 
$T_{L/K}\eta = \eta + \eta^{\sigma} \equiv  \eta - \eta^\sigma \bmod 2$ 
for every $\eta \in \OO_L$. Therefore, 
$N_{L/K}(\xi-\eta) = \left(\xi-\frac12T_{L/K}\eta\right)^2  
-(\eta - \eta^\sigma)^2$ for every $\xi \in K$, and so 
$N_{L/\Q}(\xi-\eta) \ge N_{K/\Q}(\xi-\frac12T_{L/K}\eta)^2 \ge M(\xi,K)^2$. 
This proves the claim.
\end{proof}

\section{Normal quartic CM-fields}

In this section we will prove that if $K$ is a normal quartic Euclidean 
CM-field, then $K$ is one of the fields listed in Theorem 1 or 2.

\subsection{Cyclic Fields} Suppose first that $K$ is a cyclic 
complex quartic number field; if $K$ is Euclidean, its class 
number is $1$ and according to Setzer, its conductor belongs 
to the set $\{5, 13, 16, 29, 37, 53, 61\}$.
         
   The field with conductor $\mathfrak f = 16$ is 
$K = \Q\big(\sqrt{-2+\sqrt{2}}\,\big)$; it has fundamental 
unit $\eps = 1+ \sqrt{2}$. Therefore, the residue class 
$1 + \sqrt{-2+\sqrt{2}} \bmod 2$ does not contain units; since (3) 
is inert in $K/\Q$, it does not contain an element of norm $3$. 
The primes $5, 7, 11, 13$ do not split completely in $K/\Q$, so
there are no elements in $\OO_K \setminus E_K$ with odd norms 
$< 2^4$: this shows that $K$ is not Euclidean.

   Next we apply Proposition 1 to $K/\Q$  with $\p = p\Z$ and with 
the values of $\alpha, \beta$ given in the following table:

$$\begin{array}{l|rrrr} 
 p      & 29 & 37 & 53 & 61 \\ \hline
 \alpha &  6 & 14 & 15 &  4 \\ 
 \beta  & 20 & 10 & 10 & 12 \\
\end{array} $$

   In order to show that $\beta$ is not a norm in $K/\Q$
($\beta-p$ is never norm because norms from $K$ are always positive), 
just notice that $(2/p) = -1$. 

We remark that it is easy to prove Thm. \ref{T1} without making
use of Setzer's results: if a cyclic quartic complex field $L$
has odd class number, then its conductor must be a prime power. 
Since the quadratic fields with prime power discriminant
$> 73$ are not norm Euclidean, Prop. \ref{P3} shows that
any norm Euclidean $L$ with conductor $p > 73$ must 
satisfy $p = N_{K/\Q} \disc(L/K) < 4^2 = 16$. This contradiction
shows that $f \le 73$. Now we compute the class numbers for
the fields in this finite list and continue as above.

\subsection{Bicyclic Fields}
Next we will deal with bicyclic fields. Let $D(m,n)$ denote the 
ring of integers in $\Q(\sqrt{m}, \sqrt{n}\,)$ and suppose that 
$D(m,n)$ is Euclidean. We will distinguish the following cases:

\subsection*{I. $D(m,n)$ contains an ideal of norm $2$}

Since $D(m,n)$ has class number $1$, this ideal of norm $2$ is 
principal. Taking relative norms shows that each of the two complex 
quadratic subfields of $\Q(\sqrt{m}, \sqrt{n}\,)$ contains an 
element of norm $2$. The only such fields are $\Q(\sqrt{-1}\,)$, 
$\Q(\sqrt{-2}\,)$, and $\Q(\sqrt{-7}\,)$, and this leaves 
us with $D(-1, 2)$, $D(-1,7)$ and $D(-2,-7)$.

Let $R = D(-2,-7)$; we know that $2R = (\mathfrak 2_1 \mathfrak 2_2)^2$   
for prime ideals $\mathfrak 2_1,\mathfrak 2_2$  of norm $2$. If $R$ were 
Euclidean, the prime residue classes mod 
$\mathfrak m =\mathfrak 2_1^2\mathfrak 2_2$ would contain elements of odd 
norm $< 8 = N\mathfrak m$. Since the unit group is generated by $-1$ and 
$\eps = 2\sqrt{-2}+\sqrt{-7}$, the congruence 
$-1 \equiv \eps \equiv 1 \bmod 2$ shows that only the residue 
class $1 \bmod \,\mathfrak m$ contains units. Since there are no elements 
of norm $3$, $5$, or $7$ in $R$, this ring is not Euclidean.
 
\subsection*{II. $D(m,n)$ does not contain an ideal of norm $2$}

   This implies that $2$ is inert in one of the quadratic subfields 
of $K$; there are the following possibilities:
\begin{enumerate}
\item[(A)] $2$ is inert in the real subfield and ramified in the 
           complex subfields; \newline
   Let $R = D(m,n)$; we may assume that $m \equiv 2,3 \bmod 4, 
n \equiv 5 \bmod 8$ and $n>0$. At least one of the complex quadratic 
subfields contains an element of norm $2$: otherwise, both subfields 
would have class number $>1$, and since $K$ is ramified over at most 
one of them, $K$ would have non-trivial class number. Therefore, $K$ 
contains $\Q(\sqrt{-1}\,)$ or $\Q(\sqrt{-2}\,)$: the possibility 
$\Q(\sqrt{-7}\,)$ is excluded since we assumed that $m \equiv 2,3 \bmod 4$.
\item[(A.1)] $R =D(-1,n), n \equiv 5 \bmod 8$ \newline
If R is Euclidean, the residue class $\frac{1+\sqrt{-n}}{1+i} \bmod 2$ 
contains an element $\alpha \in R$ such that $N\alpha < 16$. This 
implies that $N_{K/\Q(\sqrt{-n})} \alpha \equiv \sqrt{-n} \bmod 2$. 
Therefore, $\Q(\sqrt{-n}\,)$ contains an element 
$\beta \equiv \sqrt{-n} \bmod 2$ of norm $< 16$. This shows $n<16$, 
i.e. $n \in \{5,\,13\}$.

Let $R = D(-1,13)$; we will apply Proposition 1 with 
$k=\Q(i), K=\Q(i,\sqrt{13}\,), \p=(3+2i), \alpha=2, \beta=4$. 
The only $b \in \Z\lbrack i\rbrack$ satisfying (1) and (3) are 
$b=-1+i$ and $b =1-2i$: since the prime ideals $(1-i)$ and 
$(1-2i)$ remain inert in $K$, these $b$ cannot be norms, and 
we get a contradiction.
\item[(A.2)] $R =D(-2,n), n \equiv 5 \bmod 8$ \newline
We look at the residue class $\alpha \equiv \sqrt{-2n} + \frac12(1+\sqrt n) 
\bmod 2$ instead and find that $\Z\lbrack i\rbrack$ contains an element 
$\equiv 1+\sqrt{-2n}  \bmod 2$ of norm $< 16$. Now $n \equiv 5 \bmod 8$
and $1+2n<16$ imply $n=5$.

\item[(B)] $2$ is inert in a complex subfield and ramified in the 
           real subfield; \newline
Here we may assume that $m \equiv 2,3 \bmod 4, n \equiv 5 \bmod 8$ and 
$n<0$. Apply Proposition 1 with $k=\Q(\sqrt n), K=\Q(\sqrt m,\sqrt n\,), 
\p=2\OO_K, \beta = \frac12(1+\sqrt n)$; note that $\beta$ is a square 
$\mod 2$ since $\OO_K/2\OO_K$  has order $3$. If $D(m,n)$ is Euclidean, 
$\OO_K$ must contain an element $\equiv \beta \bmod 2$ with norm $<4$; 
obviously, $\beta$ is no unit if $n<-3$, and this implies that 
$N_{k/\Q}\beta = 3$. Therefore, $n \in \{-3, -11\}$, and if 
$n=-11$, $\beta$ must be norm of an element in $D(m,n)$ with absolute 
norm $3$. Taking the relative norm to $\Q(\sqrt m\,)$ of this element 
shows that $\Z\lbrack\sqrt m\,\rbrack$ contains an element of norm $3$, 
and this gives $m=-2$. In order to show that $D(-2,-11)$ is not 
Euclidean, we apply Proposition 2 with 
$t=2, \kappa = \frac{6523}{5808}$, and 
$\xi = \xi_1 = \frac{13}{66}\sqrt{-11}(1 - \sqrt{-2}), 
 \eps = 7\sqrt{-2} + 3\sqrt{-11}$. This implies
$\xi\eps^2 \equiv \xi \bmod R$, and $\mu_1 \approx 2.41, 
\mu_2 \approx 1.71, \mu_3 \approx 0.73, \mu_4 \approx 0.52$, 
so only a few values have to be tested.

   We are left with $R=D(m,-3)$. Let $\rho$ be a primitive third 
root of unity, and let $N_m$ denote the relative norm of $K/\Q(\sqrt m)$. 
If there is an element $\alpha \equiv  \rho + \sqrt m \bmod 2$, then 
$N_m \alpha \equiv m+1+\sqrt m \bmod 2$. In case $m \equiv 2 \bmod 4$, 
this implies the existence of an element $\beta \equiv 3+\sqrt{m} \bmod 2$ 
with norm $<16$ in $\Z[\sqrt{m}]$ and this yields $|m| < 16$. The only 
domains with class number 1 among these are $D(2,-3)$ and $D(-2,-3)$. 
Similarly, in case $m \equiv 3 \bmod 4$ we find only $D(-1, -3)$.

\item[(C)] $2$ is unramified in $K$. \newline
Write $R=D(m,n)$ and assume that $m,n<0$. If R is Euclidean, then the 
residue class $ \alpha \equiv \frac12 ( 1+\sqrt m) \bmod 2$  
contains an element of norm $<16$; therefore, one of the classes 
$ \frac12 (1 \pm \sqrt m),\, \frac12 (3 \pm \sqrt m) \bmod 2$ 
contains such an element, and this implies $|m|<64$. Similarly, 
$|n| < 64$, and among the remaining $D(m,n)$, only the following 
have class number 1:

$$\begin{array}{lrll} 
      m = & -3, &n = &-7, -11, -15, -19, -43, -51; \\
      m = & -7, &n = &-11, -19, -35, -43; \\
      m = &-11, &n = &-19.  
\end{array}$$

 Applying Proposition 1 to $K = \Q(\sqrt m,\sqrt n\,)$ and 
$k=\Q(\sqrt n\,)$, we can exclude the following fields:

$$\begin{array}{r|r|ccc}
\rsp      m  &   n           &\alpha     &\p   &b \bmod \p \\ \hline
\rsp     -3  & -43  &\frac12(1+\sqrt{-43})     &(3)       &\sqrt{-43} \\
\rsp    -7  & -11        & 2+\sqrt{-11}       &(7)     &-3\sqrt{-11} \\
\rsp     -7  & -19        &  2  &\frac12(3+\sqrt{-19})       &-3      \\
\rsp     -7  & -43  &\frac12(3+\sqrt{-43})  &(7) &\frac12(-3+3\sqrt{-43}) \\
\rsp    -11  & -19        &  4     &\frac12(5+\sqrt{-19})    &5
\end{array}$$

\end{enumerate}

This takes care of the negative part of Theorem 2. In order to
prove that the fields listed there (as well as a few others, 
cf. \cite{Lem2}) are in fact Euclidean, we used programs written 
in BASIC (partial results have been obtained by Lakein \cite{Lak}).
The algorithms are described in \cite{CL} for the 
case of cubic fields; here is what is known about the Euclidean 
minima of the fields in Thm. \ref{T2}:

\medskip

{\begin{center} {\parbox{10em}{
\begin{tabular}{|r|r|c|} \hline
  m 	&  n 	  	& $M(K)$ \\ \hline
 $-1$ 	& $ -2$ 	& $1/2$  \\
	& $ -3$ 	& $1/4$  \\
 	& $  5$ 	& $5/16$ \\
	& $ -7$ 	& $1/2$  \\ \hline
 $-7$	& $  5$   	& $9/16$ \\ \hline
 $-2$	& $  5$   	& $11/16$ \\ \hline
\end{tabular}} \hfil \parbox[t]{10em}{
\begin{tabular}{|r|r|c|} \hline
  m 	&  n 	& $M(K)$ \\ \hline
 $-3$	& $  2$	& $\ge 1/4$ \\
 	& $ -2$	& $1/3$ \\
	& $  5$	& $1/4$ \\
 	& $ -7$	& $4/9$ \\
 	& $-11$	& $< 0.46$ \\
	& $ 17$	& $13/16$ \\
	& $-19$	& $< 0.95$ \\ \hline
\end{tabular}}} \end{center}}


\begin{thebibliography}{999}


\bibitem{BSD1} E.S.~Barnes, H.P.F.~Swinnerton-Dyer,
{\em The inhomogeneous minima of binary quadratic forms I},
Acta Math. {\bf 87} (1952),  259--323
%

\bibitem{BSD2} E.S.~Barnes, H.P.F.~Swinnerton-Dyer,
{\em The inhomogeneous minima of binary quadratic forms II},
Acta Math. {\bf 88} (1952),  279--316
%

    
\bibitem{Cas} J. W. S. Cassels,
{\em The inhomogeneous minima of binary quadratic, ternary cubic,
       and quaternary quartic forms},
Proc. Cambridge Phil. Soc. {\bf 48} (1952), 519--520
%

\bibitem{CL} S.~Cavallar, F.~Lemmermeyer,
{\em The Euclidean algorithm in cubic number fields}, 
Proc. Number Theory Eger 1996 (1998), 123--146
%

\bibitem{Cio} V. Cioffari,
{\em The Euclidean condition in pure cubic and 
complex quartic fields}, Math. Comp. {\bf 33} (1979), 389--398
%
\bibitem{Lak} R. B. Lakein,
{\em Euclid's algorithm in complex quartic fields},
Acta Arithm. {\bf 20} (1972) 393--400
%

\bibitem{Lem1} F.~Lemmermeyer,
{\em Euklidische Ringe},
Diplomarbeit  Univ. Heidelberg, 1989 
%

\bibitem{Lem2} F.~Lemmermeyer,
{\em The Euclidean Algorithm in Algebraic Number Fields},
Expo. Math. {\bf 13} (1995),  385--416
%

\bibitem{vdL} F. J. van der Linden,
{\em Euclidean rings of integers of fourth degree fields},
Lecture notes in Math. {\bf 1068} (1983), 139--148
%


\bibitem{Sau} J. Sauvageot,
{\em Algorithmes d'Euclide dans certains corps biquadratiques},
Sem. Delange-Pisot-Poitou (1972/73)
%

\end{thebibliography}
\enddocument